\documentclass[twoside, 11pt]{amsart}
\usepackage{amsmath, amssymb, amsthm}

\newtheorem{thm}{Theorem}
\newtheorem{cor}[thm]{Corollary}
\newtheorem{dfn}[thm]{Definition}
\newtheorem{lem}[thm]{Lemma}
\newtheorem{pro}[thm]{Proposition}

\def\ad{{\rm ad}}
\def\C{{\mathbb C}}
\def\Diff{{\rm Diff}}
\def\id{{\rm id}}
\def\Lie{{\mathcal L}}
\def\O{{\mathcal O}}
\def\R{{\mathbb R}}
\def\Z{{\mathbb Z}}

\title[Properly essential conformal groups]{On properly essential classical conformal diffeomorphism groups}
\author[S.~M\"uller \& P.~Spaeth]{Stefan M\"uller and Peter Spaeth}
\email{mueller@kias.re.kr \text{\it{and}} spaeth@kias.re.kr}
\address{Korea Institute for Advanced Study, Seoul 130-722, Republic of Korea}


\begin{document}
\thispagestyle{plain}

\begin{abstract}
We prove that various classical conformal diffeomorphism groups, which are known to be essential \cite{banyaga:ecg00}, are in fact properly essential.
This is a consequence of a local criterion on a conformal diffeomorphism in the form of a cohomological equation.
Furthermore, we study the orbit of a tensor field under the action of the conformal diffeomorphism group for these classical conformal structures.
On every closed contact manifold, we find conformal contact forms that are not diffeomorphic.
\end{abstract}

\maketitle
\section{Introduction and statement of main results}
Following A.~Banyaga~\cite{banyaga:ecg00}, we consider geometrical structures defined as (oriented) conformal classes $\sigma$ of some tensor field $\tau$ on a smooth manifold $M$, where two such tensor fields $\tau$ and $\tau'$ are considered equivalent if and only if $\tau' = e^f \tau$ for some smooth function $f$ on $M$.
We write $\sigma = [\tau]$ for the \emph{(oriented) conformal structure}, or $\tau \in \sigma$ if $\sigma = [\tau]$, and $(M,\sigma)$ for the \emph{(oriented) conformal manifold}.
The conformal structure $\sigma = [\tau]$ is an orientation of $M$ if $\tau$ is a volume form, a conformal symplectic structure if $\tau$ is a symplectic form, and a contact structure if $\tau$ is a contact form.
In any of these three cases, $\sigma$ is called an \emph{(oriented) classical conformal structure}.

For $\tau \in \sigma$ we denote by
	\[ C_{\ad,+}^\infty (M,\tau) = \{ f \in C^\infty (M) \mid e^f \tau \in \sigma \} \]
the set of \emph{admissible conformal rescaling factors} of the given tensor field $\tau$.
If $\tau' = e^g \tau$, there is a natural $1-1$ correspondence between $C_{\ad,+}^\infty (M,\tau)$ and $C_{\ad,+}^\infty (M,\tau')$ given by $f \mapsto f - g$.
If $\tau$ is a contact form or a volume form, then $C_{\ad,+}^\infty (M,\tau) = C^\infty (M)$, whereas if $\tau$ is a symplectic form on a manifold of dimension greater than two, then by the closedness condition, $C_{\ad,+}^\infty (M,\tau)$ consists of the constant functions only.

The group $\Diff (M)$ of diffeomorphisms of $M$ acts on conformal classes of tensor fields.
Indeed, if $\tau' = e^f \tau$, and $\varphi \in \Diff (M)$, then $\varphi^* \tau' = e^{f \circ \varphi} \varphi^* \tau$, and $\varphi^* \tau$ and $\varphi^* \tau'$ define the same (oriented) conformal structure.
We denote by
\begin{eqnarray*}
	\Diff_+ (M,\sigma) &=& \{ \varphi \in \Diff (M) \mid \varphi^* \sigma = \sigma \} \\
	&=& \{ \varphi \in \Diff (M) \mid \varphi^* \tau = e^f \tau \ {\rm for \ some} \ f \in C_{\ad,+}^\infty (M,\tau) \}
\end{eqnarray*}
the group of automorphisms (or \emph{conformal (diffeomorphism) group}) of the oriented conformal structure $\sigma = [\tau]$, and write
	\[ C_+^\infty (M,\tau) = \{ f \in C^\infty (M) \mid \exists \varphi \in \Diff_+ (M,\sigma) \ {\rm such \ that} \ \varphi^* \tau = e^f \tau \}, \]
for the set of \emph{conformal factors} of those diffeomorphisms that preserve the oriented conformal structure $\sigma = [\tau]$.
If $\tau' = e^g \tau$, there is a natural $1-1$ correspondence between $C_+^\infty (M,\tau)$ and $C_+^\infty (M,\tau')$ given by $f \mapsto f + g \circ \varphi - g$ provided $\varphi^* \tau = e^f \tau$.
If $\tau \in \sigma$, we also consider the group
	\[ \Diff (M,\tau) = \{ \varphi \in \Diff (M) \mid \varphi^* \tau = \tau \} \subset \Diff_+ (M,[\tau]) \]
of automorphisms of the tensor field $\tau$.

\begin{dfn}
The (oriented) conformal diffeomorphism group $\Diff_+ (M,\sigma)$ is called \emph{inessential} if there exists $\tau \in \sigma$ such that $\Diff (M,\tau) = \Diff_+ (M,\sigma)$, and \emph{essential} otherwise.
The group $\Diff_+ (M,\sigma)$ is called \emph{properly essential} if
\begin{equation} \label{eqn:essential}
	\bigcup_{\tau \in \sigma} \Diff (M,\tau) \subsetneq \Diff_+ (M,\sigma).
\end{equation}
\end{dfn}

In the article \cite{banyaga:ecg00}, Banyaga proved that many classical conformal structures carry essential conformal groups.
We show that their conformal groups are in fact properly essential.
More precisely, we prove the following theorem.

\begin{thm} \label{thm:essential}
The conformal diffeomorphism groups of the following conformal manifolds $(M,\sigma)$ are properly essential.

(i) $(M,\xi)$ a contact manifold,

(ii) $(M,[\omega])$ a conformal symplectic manifold that supports a complete Liouville vector field, and

(iii) $(M,\sigma)$ an oriented manifold.
\end{thm}

The same statement holds for the identity components of these conformal diffeomorphism groups.
We also discuss compact symplectic manifolds as well as more general noncompact symplectic manifolds in Section~\ref{sec:symplectic}.
The condition in part (ii) of the theorem means that $M$ admits a smooth vector field $X$ with $\Lie_X \omega = \omega$, that in addition has a well-defined flow everywhere on $M$ for all times.
This list includes for example $T^* N$, where $N$ is a smooth manifold, with its standard symplectic form, and in particular $\R^{2 n} = T^* \R^n$, as well as Stein manifolds, and the symplectizations of contact manifolds.

We also consider the orbit of a tensor field $\tau \in \sigma$ under the action of the conformal diffeomorphism group $\Diff_+ (M,\sigma)$, that is
	\[ \O (\tau) = \{ \tau' \mid \tau' = \varphi^* \tau \ {\rm for \ some} \ \varphi \in \Diff_+ (M,\sigma) \}. \]
Then $\O (\tau) \cong \Diff_+ (M,\sigma) / \Diff (M,\tau)$, and moreover, $\O (\tau)$ can be identified naturally with the subset $C_+^\infty (M,\tau) \subset C_{\ad,+}^\infty (M,\tau)$ via the relation $\tau' = e^f \tau$.

For noncompact manifolds $M$, we say that the orbit $\O (\tau)$ of $\tau$ is \emph{maximal} if $C_+^\infty (M,\tau) = C_{\ad,+}^\infty (M,\tau)$.
If $M$ is compact, and $\tau$ induces a volume form on $M$, every function $f \in C_+^\infty (M,\tau)$ satisfies an obvious average value constraint, which will be discussed in detail case by case below.
We say that $\O (\tau)$ is maximal if $C_+^\infty (M,\tau)$ consists of all smooth functions in $C_{\ad,+}^\infty (M,\tau)$ that satisfy this necessary constraint on the average value with respect to the induced volume form.
The notion of maximality depends only on the conformal class of $\tau$.
Indeed, suppose $\tau$ and $\tau' = e^g \tau \in \sigma$, and that $\O (\tau)$ is maximal.
Then there exists $\varphi \in \Diff_+ (M,\sigma)$ such that $\varphi^* \tau = \tau'$.
Moreover, if $f \in C_{\ad,+}^\infty (M,\tau')$, then $f + g \in C_{\ad,+}^\infty (M,\tau)$, and if $f$ satisfies the constraint on the average value with respect to $\tau'$, then $f + g$ satisfies the same constraint with respect to $\tau$.
By the maximality hypothesis, there exists $\psi \in \Diff_+ (M,\sigma)$ such that $\psi^* \tau = e^{f + g} \tau$.
Then $(\varphi^{-1} \circ \psi)^* \tau' = e^f \tau'$, proving $\O (\tau')$ is maximal.

If $\mu$ is a volume form on a closed manifold, it is well-known that $\O (\mu)$ is maximal by Moser's argument.
The same holds for open manifolds $M$ provided every vector field on $M$ is complete, due to the nondegeneracy of the volume form and vanishing of the top dimensional cohomology.
Moreover, Banyaga~\cite[Section 3.4]{banyaga:ecg00} proved that if $(M,\omega)$ is a symplectic manifold supporting a complete Liouville vector field, then the symplectic form $\omega$ has maximal orbit.
For $(M,\omega)$ a compact symplectic manifold, the notion of maximality is meaningless, since $C_+^\infty (M,\omega) = C_{\ad,+}^\infty (M,\omega) = \{ 0 \}$.
For closed contact manifolds, we will demonstrate the following theorem.
Under a certain additional hypothesis, the same statement holds for general contact manifolds.
See Section~\ref{sec:contact} for details.

\begin{thm} \label{thm:orbit}
For $(M,\xi)$ a closed contact manifold, and $\alpha$ any contact form, the orbit $\O (\alpha)$ is not maximal.
\end{thm}

Our proofs differ from the methods employed by Banyaga, who proved the essentiality of a conformal group to be equivalent to the non-vanishing of a certain conformal invariant.
This non-vanishing in turn is often a consequence of the non-triviality of $\O (\tau)$, together with transitivity properties of the group $\Diff_+ (M,\sigma)$.
In this work, we find local obstructions to equation (\ref{eqn:essential}) in order to prove our stronger result Theorem~\ref{thm:essential}.
In Section~\ref{sec:obstruction} we single out the obstruction for a conformal diffeomorphism to preserve any of the tensor fields in a conformal class.
We then discuss contact structures, conformal symplectic structures, and orientations in Sections~\ref{sec:contact}, \ref{sec:symplectic}, and \ref{sec:volume}, respectively.
The relation of Banyaga's conformal invariant to the property of being properly essential is discussed in Section~\ref{sec:banyaga-inv}.
In Section~\ref{sec:orientation} we consider conformal factors that can be any non-vanishing function, and allow the possibility of diffeomorphisms reversing the orientation of the conformal class $\sigma$, in the sense that $\varphi^*\tau = f \tau$ for some negative smooth function $f$.
The last section contains some final remarks and defines a new invariant of a contact diffeomorphism.
This paper originated with the problems solved in Section~\ref{sec:contact}.

\section{Diffeomorphisms preserving a conformal structure} \label{sec:obstruction}
Let $(M,\sigma)$ be a conformal manifold, and $\varphi$ a conformal diffeomorphism.
If $\tau \in \sigma$, then there exists a smooth function $f$ on $M$ such that $\varphi^* \tau = e^f \tau$.
Suppose $\tau' \in \sigma$ is any other tensor field in the conformal class of $\tau$, then $\tau' = e^g \tau$ for a smooth function $g$ on $M$.
The diffeomorphism $\varphi$ preserves $\tau'$, that is, $e^g \tau = \varphi^* (e^g \tau) = e^{g \circ \varphi + f} \tau$, if and only if
\begin{equation} \label{eqn:criterion}
	f = g - g \circ \varphi.
\end{equation}
That is, $\varphi$ preserves a tensor field in the conformal class $\sigma$ if and only if there exists a solution $g \in C_{\ad,+}^\infty (M,\tau)$ to the \emph{cohomological equation} (\ref{eqn:criterion}).
Thus we have the following lemma, which provides a local criterion for a conformal diffeomorphism to not preserve any tensor field in the given conformal class, and in particular, for the conformal group to be properly essential.

\begin{lem} \label{lem:criterion}
Let $\varphi \in \Diff_+ (M,\sigma)$ be a conformal diffeomorphism, $\tau \in \sigma$, and $\varphi^* \tau = e^f \tau$.
Suppose $\varphi$ has a fixed point at which $f$ does not vanish.
Then $\varphi \notin \Diff (M,\tau')$ for any $\tau' \in \sigma$, and $\Diff_+ (M,\sigma)$ is properly essential.
\end{lem}

\begin{proof}
If $x \in M$ is a fixed point of $\varphi$ with $f (x) \not= 0$, then $g (x) - g \circ \varphi (x) = 0$, and by equation (\ref{eqn:criterion}), $\varphi$ does not preserve any $\tau' \in \sigma$.
So $\varphi$ is contained in the right-hand side of equation (\ref{eqn:essential}), but not in the set on the left-hand side.
\end{proof}

If $\varphi$ is as in the lemma, its conformal factor with respect to $\tau' = e^g \tau \in \sigma$ equals $f + g \circ \varphi - g$.
In particular, at $x \in M$ this conformal factor coincides with $f (x) \not= 0$ for all $\tau' \in \sigma$.
Therefore, if $\sigma$ is the conformal class of a contact form, a symplectic form, or a volume form, then $\varphi$ does not preserve the volume form at the point $x$ induced by $\tau'$ for all $\tau' \in \sigma$.
On the other hand, every diffeomorphism in $\Diff (M,\tau')$ must preserve that volume form.
We obtain another criterion for a conformal group to be properly essential, which is more general (but less applicable) than the previous one if $\sigma$ is a classical conformal structure.

\begin{lem}
Let $\sigma$ be the conformal class of a contact form, a symplectic form, or a volume form, and $\varphi \in \Diff_+ (M,\sigma)$ a conformal diffeomorphism, such that for all $\tau \in \sigma$, $\varphi$ does not preserve the volume form induced by $\tau$.
Then $\Diff_+ (M,\sigma)$ is properly essential.
\end{lem}

\section{Contact structures} \label{sec:contact}
Consider $\R^{2 n + 1}$ with coordinates $(x, y, z) = (x_1, \ldots, x_n, y_1, \ldots, y_n, z)$, and its standard contact structure $\xi = \ker \alpha$, where $\alpha$ is the standard contact form $dz - \sum_{i = 1}^n y_i dx_i$, so that the contact vector field of a smooth function $H$ on $\R^{2 n + 1}$ is given by
	\[ X_H = \sum_{i = 1}^n \left( - \frac{\partial H}{\partial y_i} \right) \frac{\partial}{\partial x_i} + \sum_{i = 1}^n \left( \frac{\partial H}{\partial x_i} + y_i \frac{\partial H}{\partial z} \right) \frac{\partial}{\partial y_i} + \left( H - \sum_{i = 1}^n y_i \frac{\partial H}{\partial y_i} \right) \frac{\partial}{\partial z} .\]
Let $H (x, y, z) = z - \sum_{i = 1}^n x_i y_i$.
Its contact flow is $\varphi_H^t (x, y, z) = (e^t x, y, e^t z)$, and $(\varphi_H^t)^* \alpha = e^t \alpha$.
The contact diffeomorphisms $\varphi_H^t$ all fix the points $(0, y, 0)$, and thus equation (\ref{eqn:criterion}) evaluated at any of the points $(0, y, 0)$ reads $t = 0$.
Thus for $t > 0$, there exists no solution $g$ to equation (\ref{eqn:criterion}), and consequently, the contact diffeomorphism $\varphi_H^t$ does not preserve any contact form that determines $\xi$.
Alternatively, it is apparent that for $t > 0$, $\varphi_H^t$ expands any volume form at the points $(0, y, 0)$.
This can also be seen directly; a rectangular solid centered there expands to a larger rectangular solid.
But a diffeomorphism preserving a contact form say $\alpha'$ also preserves the volume form $\alpha' \wedge (d\alpha')^n$ induced by it.
The same argument works for the Hamiltonian $F (x, y, z) = 2 z - \sum_{i = 1}^n x_i y_i$ with flow $\varphi_F^t (x, y, z) = (e^t x, e^t y, e^{2 t} z)$ and fixed point $(0, 0, 0)$.
Cutting off $H$ or $F$ in a neighborhood of a fixed point and using Darboux's theorem produces a diffeomorphism satisfying the hypothesis of Lemma~\ref{lem:criterion} on any contact manifold $(M^{2 n + 1},\xi)$.
That proves Theorem~\ref{thm:essential} (i).

Given two contact forms $\alpha$ and $\alpha' = e^f \alpha$, we consider the problem of the existence of a diffeomorphism $\varphi$ with $\varphi^* \alpha = \alpha'$.
If $M$ is compact, an obvious necessary condition on the function $f \in C^\infty (M)$ is given by the following lemma.
Denote by
\begin{equation} \label{eqn:normalized-volume}
	\mu = \frac{1}{\int_M \alpha \wedge (d\alpha)^n} \cdot \alpha \wedge (d\alpha)^n
\end{equation}
the normalized canonical volume form on $M$.

\begin{lem}
Let $(M^{2 n + 1},\xi)$ be a compact contact manifold, and $\alpha$ a contact form.
If $e^f \alpha$ is diffeomorphic to $\alpha$, i.e.\ there exists a contact diffeomorphism $\varphi$ with $\varphi^* \alpha = e^f \alpha$, then the $L^{n + 1}$-norm of $e^f$ is equal to $1$, or in other words, $e^{(n + 1) f}$ has average value $1$, with respect to the volume form $\mu$.
In particular, if the smooth function $f$ has positive average value with respect to $\mu$, then $e^f \alpha$ is not diffeomorphic to $\alpha$.
The same conclusion holds if $f \le 0$ everywhere, and $f$ is negative at some point.
\end{lem}

\begin{proof}
The first statement follows from the change of variables formula applied to $\varphi^* (\alpha \wedge (d\alpha)^n) = e^{(n + 1) f} \alpha \wedge (d\alpha)^n$.
Jensen's inequality applied to the exponential function shows that $f$ has nonpositive average value.
On the other hand, $(\varphi^{-1})^* \alpha = e^{- f \circ \varphi^{-1}} \alpha$, so that $- f \circ \varphi^{-1}$ also has nonpositive average value, which is impossible if $f \le 0$ but not identically zero.
\end{proof}

In contrast to the situation for volume forms, this necessary condition is not sufficient.
The proof requires a series of lemmata.

\begin{lem} \label{lem:conjugate}
If two contact forms $\alpha$ and $\alpha'$ are diffeomorphic, then their Reeb flows $\varphi_R^t$ and $\varphi_{R'}^t$ are conjugate.
\end{lem}

\begin{proof}
If $\alpha' = \varphi^* \alpha$, then $\varphi_* R' = R$, and thus $\varphi \circ \varphi_{R'}^t \circ \varphi^{-1} = \varphi_R^t$.
\end{proof}

\begin{lem} \label{lem:closed-orbit}
Every closed contact manifold $(M,\xi)$ supports a contact form that has a closed Reeb orbit with trivial normal bundle.
\end{lem}

Since we have not been able to find a proof of this fact anywhere, we give a proof here.

\begin{proof}
We prove the lemma by induction on $n$, where $2 n + 1$ is the dimension of $M$.
The degenerate case $n = 0$ is trivial.
Choose a contact form $\alpha$ on $(M,\xi)$ that is carried by an open book decomposition with binding $K \subset M$.
By the induction hypothesis, the contact structure $\xi |_K$ on $K$ induced by the contact form $\alpha |_K$ supports a contact form that has a closed Reeb orbit with trivial normal bundle.
Thus there exists a smooth function $f$ on $K$ such that $e^f \alpha |_K$ has such a closed Reeb orbit $\gamma$.
Since the normal bundle of $K$ in $M$ is trivial, $f$ extends to a smooth function $g$ on $M$, and $e^g \alpha$ is a contact form on $M$ with kernel $\xi$ that restricts to $e^f \alpha |_K$ on $K$.
As in the proof of the contact neighborhood theorem, see for example \cite{geiges:ict08}, we can find a smooth function $h$ on $M$ that vanishes on $K$, such that the Reeb flow of $e^{g + h} \alpha$ restricts to the Reeb flow of $e^f \alpha |_K$ on $K$.
Then $\gamma$ is the desired closed Reeb orbit of the contact form $e^{g + h} \alpha$ on $M$.
\end{proof}

\begin{lem} \label{lem:standard}
Let $\alpha$ be a contact form and $\gamma$ be a closed Reeb orbit with trivial normal bundle.
Then there exists a contact form $\alpha'$ that induces the same contact structure as $\alpha$ and that is standard near $\gamma$.
In particular, all Reeb orbits of $\alpha'$ near $\gamma$ are closed and have the same period as $\gamma$.
\end{lem}

\begin{proof}
On the normal bundle of $\gamma$, consider the contact structure $\xi_0$ induced by the standard contact form $\alpha_0 = d\theta - \sum_i y_i dx_i$, where the coordinate $\theta$ parameterizes $\gamma$, and $x_i$ and $y_i$ are normal coordinates.
By the contact neighborhood theorem, $\xi_0$ is diffeomorphic near $\gamma$ to the contact structure induced by $\alpha$.
Therefore there exists a smooth function $f$, locally defined in a neighborhood of $\gamma$, so that $e^f \alpha$ is locally diffeomorphic to $\alpha_0$.
Since $\gamma$ has trivial normal bundle, $f$ extends to a globally defined smooth function $g$, and the contact form $\alpha' = e^g \alpha$ is locally diffeomorphic to $\alpha_0$ near $\gamma$.
\end{proof}

\begin{lem}
Every closed contact manifold $(M^{2 n + 1},\xi)$ with contact form $\alpha$ supports another contact form not diffeomorphic to $\alpha$.
If $M$ is compact, this contact form can be chosen so that its induced volume form has the same total volume as the volume form induced by $\alpha$.
\end{lem}

\begin{proof}
Let $\alpha'$ be a contact form for $\xi$ as in the previous lemma.
For a generic contact form $\alpha''$, its closed Reeb orbits of any fixed period are isolated as smooth maps $S^1 \to M$.
Thus by Lemma~\ref{lem:conjugate}, $\alpha''$ cannot be diffeomorphic to $\alpha'$.
The contact form $\alpha$ on $(M,\xi)$ may be diffeomorphic to one of the above contact forms $\alpha'$ or $\alpha''$, but not both.
If $M$ is compact, we can modify $\alpha'$ and $\alpha''$ to induce volume forms of total volume equal to the total volume of the volume form induced by $\alpha$.
If $\alpha$ is generic in the above sense, write $\alpha' = e^f \alpha$, and choose a nonempty open subset $U$ of $M$ that does not intersect the neighborhood of $\gamma$ filled out by closed Reeb orbits of $\alpha'$, and a function $g$ compactly supported inside $U$, such that $e^{(n +1) (f + g)}$ has average value $1$ with respect to the volume form $\mu$ in equation (\ref{eqn:normalized-volume}) induced by $\alpha$.
But $e^{(n +1) (f + g)} \alpha \wedge (d\alpha)^n$ is the volume form induced by the contact form $e^{f + g} \alpha$, and the Reeb flow of the latter coincides with the Reeb flow of $e^f \alpha$ in a neighborhood of $\gamma$.
Thus $\alpha$ and $e^{f + g} \alpha$ are not diffeomorphic.
Otherwise, observe that if $\alpha'' = e^h \alpha$ is generic, then so is $e^{h + c} \alpha$ for any real constant $c$.
This constant can be chosen so that $e^{(n + 1) (h + c)}$ has the correct average value.
\end{proof}

Thus the orbit of any contact form $\alpha$ on a closed contact manifold is not maximal, and the proof of Theorem~\ref{thm:orbit} is complete.
It would be interesting to have a complete understanding of the orbit $\O (\alpha)$ of a contact form $\alpha$.
This problem seems to be very difficult.

In contrast, by Darboux's theorem, locally near a point $x$ all contact forms on a contact manifold $M$ are diffeomorphic.
These local diffeomorphisms can be constructed using Moser's argument so that they coincide with the identity map near the boundary of a larger local chart, and thus extend to global diffeomorphisms of $M$ that exchange the two contact forms near $x \in M$.

Note however that these Darboux neighborhoods do not contain any closed Reeb orbits, since a contact form on an open contact manifold may not admit any, as is the case for the standard contact form on $\R^{2 n + 1}$.
In general one may have to consider other invariants of the contact form in order to distinguish contact forms in the same conformal class on open manifolds.
However, the preceding lemmata also prove the following result.

\begin{pro}
If $(M,\xi)$ admits a contact form having a closed Reeb orbit with trivial normal bundle, and $\alpha$ is any contact form, then the orbit $\O (\alpha)$ is not maximal.
\end{pro}

\section{Conformal symplectic structures} \label{sec:symplectic}
Let $\omega$ be a symplectic form on a smooth manifold $M$ of necessarily even dimension $2 n$, and $\sigma$ its (oriented) conformal class.
If $n = 1$, $\omega$ is just an area form, and $\sigma$ the induced orientation on the surface $M$.
We assume henceforth that the dimension $2 n \ge 4$, and postpone the dimension two case to the next section.
Due to the closedness assumption, any other symplectic structure in the same conformal class is of the form $e^\lambda \omega$ for a real constant $\lambda$.
In other words, $C_{\ad,+}^\infty (M,\omega') = \R$ for any $\omega' \in [\omega]$.

Suppose first that $M$ is compact.
Let $\varphi \in \Diff_+ (M,\sigma)$ be a conformal symplectic diffeomorphism, and $\varphi^* \omega = e^\mu \omega$ for a real constant $\mu$.
By the change of variables formula, $n \mu = 0$, and thus $\mu = 0$.
That means every conformal diffeomorphism of $\sigma$ preserves $\omega$, and in fact, any symplectic form $e^\lambda \omega \in \sigma$.
Thus if $M$ is compact, $\Diff_+ (M,\sigma)$ is trivially inessential.

Now suppose that $M$ is noncompact.
Since all the functions $f$ on the left-hand side of equation (\ref{eqn:criterion}) are constant, it suffices to find one conformal diffeomorphism $\varphi$ with $\mu \not= 0$.

\begin{lem} \label{lem:expand}
If there exists a conformal diffeomorphism $\varphi \in \Diff_+ (M,\sigma)$ with $\varphi^* \omega = e^\mu \omega$ and $\mu \not= 0$, then $\Diff_+ (M,\sigma)$ is properly essential, and $\Diff_+ (M,\sigma)$ is inessential otherwise.
In particular, if $\Diff_+ (M,\sigma)$ is essential it is properly essential.
\end{lem}

Suppose that $M$ supports a complete Liouville vector field $X$, i.e.\ $\Lie_X \omega = \omega$.
Note that then $M$ is necessarily noncompact.
By our assumption, $X$ integrates to a globally defined flow $\varphi_t$, and the identity $\varphi_t^* \omega = e^t \omega$ holds, or $\varphi_t$ has conformal factor $t \in \R$.

Let $M = T^*N$ with coordinates $q \in M$ and $p$ in the fiber direction, and its standard symplectic form $dp \wedge dq = \sum_i dp_i \wedge dq_i$.
Then the smooth vector field $p \, \partial / \partial p$ is a complete Liouville vector field.
This includes the case of $\R^{2 n} = T^* \R^n$ with its standard symplectic structure.
If $M$ is a Stein manifold, it also admits a globally defined complete Liouville vector field \cite{banyaga:ecg00}.
For $M \times \R$ the symplectization of a contact manifold $M$ with symplectic form $- d (e^\theta \alpha)$, where $\alpha$ is a contact form on $M$ and $\theta$ the variable in the $\R$ direction, the vector field $\partial / \partial \theta$ is complete Liouville.
Theorem~\ref{thm:essential} (ii) is proved.

Since the conformal factor $\mu$ is globally constant, diffeomorphisms as in Lemma~\ref{lem:expand} can of course not be constructed locally as in the case of contact structures in the previous section.
Global constructions are most easily achieved as time-one maps of complete vector fields, and this is precisely the situation of Theorem~\ref{thm:essential} (ii).
As a partial converse, we have the following lemma.
Recall that by Lemma~\ref{lem:expand}, the hypothesis implies that $\Diff_+ (M,\sigma)$ is properly essential.

\begin{lem}
Suppose there exists $\varphi \in \Diff_0 (M,\sigma) \subset \Diff_+ (M,\sigma)$, such that $\varphi_* \omega = e^\mu \omega$ for a (and thus any) $\omega \in \sigma$, with nonzero conformal factor $\mu \in \R$.
Then $(M,\omega)$ supports a Liouville vector field, and in particular, $\omega$ is exact.
\end{lem}

\begin{proof}
Choose an isotopy $\varphi_t$ from $\varphi_0 = \id$ to $\varphi_1 = \varphi$ with $\varphi_t^* \omega = e^{\mu_t} \omega$.
If $X_t$ denotes its infinitesimal generator, then
	\[ \varphi_t^* (d (\iota_{X_t} \omega)) = \varphi_t^* (\Lie_{X_t} \omega) = \frac{d}{dt} (\varphi_t^* \omega) = (\frac{d}{dt} \mu_t) \cdot e^{\mu_t} \omega = \varphi_t^* ((\frac{d}{dt} \mu_t) \cdot \omega). \]
Since $\mu \not=0$, the map $t \mapsto \mu_t$ cannot be constant everywhere.
After rescaling, we obtain a Liouville vector field on $(M,\omega)$.
\end{proof}

\begin{cor}
If $\omega$ is not exact, then $\Diff_0 (M,\sigma)$ is inessential.
\end{cor}

Note that the last part of the lemma as well as its corollary also follow from the action of $\Diff_+ (M,\sigma)$ on the second deRham cohomology $H^2 (M,\R)$.
More generally, for non-exact $\omega$, Lemma~\ref{lem:expand} can be restated as follows.

\begin{lem}
If the symplectic manifold $(M,\omega)$ is not exact, then $\Diff_+ (M,\sigma)$ is inessential if the action (by multiplication) of $\Diff_+ (M,\sigma)$ on $H^2 (M,\R)$ fixes the cohomology class of $\omega$, and it is properly essential otherwise.
\end{lem}

\section{Orientations} \label{sec:volume}
Let $M = \R^n$ equipped with the standard volume form $\mu = dx_1 \wedge \cdots \wedge dx_n$, and $\sigma$ the induced orientation on $\R^n$.
Let $\rho \colon \R \to \R$ be a smooth function with $\rho (r) = r / n$ near $r = 0$, and define a smooth vector field $X$ on $\R^n$ by $\sum_i \rho (x_i) \partial / \partial x_i$.
Then $\Lie_X \mu = \lambda \mu$ with $\lambda = \sum_i \rho' (x_i)$, and since $\mu$ is a volume form, the flow $\varphi_t$ of $X$ satisfies $\varphi_t^* \mu = e^{f_t} \mu$, where $f_t = \int_0^t \lambda \circ \varphi_s \, ds$.
The flow $\varphi_t$ has a fixed point at the origin $0 \in \R^n$, and $f_t (0) =  t$.
Applying Lemma~\ref{lem:criterion} to the diffeomorphism $\varphi_t$ for any $t \not= 0$ proves $\Diff_+ (M,\sigma)$ is properly essential.

Now let $M^n$ be any oriented smooth manifold.
After cutting off the vector field $X$, the diffeomorphism $\varphi_1$ is supported in the unit ball $B \subset \R^n$.
Thus for any local chart $\psi \colon U \to V \subset M$ with $\overline{B} \subset U \subset \R^n$, the map $\psi \circ \varphi_1 \circ \psi^{-1}$ extends to a global diffeomorphism of $M$.
If $\mu$ is a volume form defining the orientation $\sigma$ of $M$, then $\psi^* \mu = v dx_1 \wedge \cdots \wedge dx_n$ for some smooth function $v$ on $U$.
The diffeomorphism $\psi \circ \varphi_1 \circ \psi^{-1}$ has a fixed point at $\psi (0)$, and its conformal factor at that point equals $1$.
Again by Lemma~\ref{lem:criterion}, $\Diff_+ (M,\sigma)$ is properly essential.
The proof of Theorem~\ref{thm:essential} (iii) is now complete.

\section{Banyaga's conformal invariant} \label{sec:banyaga-inv}
Let $\sigma$ be a conformal structure on a smooth manifold $M$, and $\tau \in \sigma$.
For $\varphi \in \Diff_+ (M,\sigma)$, denote by $f = f_\varphi$ the unique smooth function on $M$ given by $\varphi^* \tau = e^f \tau$.
Banyaga defined a conformal invariant of $(M,\sigma)$ by the cohomology class of the cocycle $D_\tau \colon \Diff_+ (M,\sigma) \to C^\infty (M)$, $D_\tau (\varphi) = f_{\varphi^{-1}}$, in $H^1 (\Diff_+ (M,\sigma), C^\infty (M))$.
The conformal group $\Diff_+ (M,\sigma)$ is essential if and only if this conformal invariant does not vanish \cite{banyaga:ecg00}.

We note that
	\[ D_\tau ((\varphi^{-1})^k) = \sum_{i = 0}^{k - 1} f_\varphi \circ \varphi^k. \]
The relation of Banyaga's conformal invariant to the property of being properly essential is the following.
We begin with a special case of a theorem of W.~H.~Gottschalk and G.~A.~Hedlund, which is most useful when $M$ is compact.

\begin{thm}{\cite[Theorem 14.11]{gottschalk:td55}}
Let $\varphi$ be a homeomorphism of $M$, $f$ a continuous function on $M$, and $N \subset M$ a compact minimal set under $\varphi$.
That means $N$ is closed and nonempty, and each orbit $\O (x) = {\{ \varphi^n (x) \mid n \in \Z \}}$, $x \in N$, is dense in $N$.
Then the following statements are equivalent.

(i) there exists $y_0 \in N$ such that the sequence (for $k$ a positive integer) of Birkhoff sums $\sum_{i = 0}^{k - 1} f_\varphi \circ \varphi^k (y_0)$ is bounded uniformly,

(ii) the sequence of Birkhoff sums $\sum_{i = 0}^{k - 1} f_\varphi \circ \varphi^k (y)$ is bounded uniformly in $y \in N$ and $k$ a positive integer,

(iii) there exists a continuous function $g$ on $N$ solving the cohomological equation $f (y) = g (y) - g \circ \varphi (y)$ for all $y \in N$.
\end{thm}

The implications $(iii) \Rightarrow (ii)$ and $(ii) \Rightarrow (i)$ are trivial, and the content of the theorem is the assertion that $(i) \Rightarrow (iii)$.
The function $g$ is unique up to an additive constant.
Indeed, if $g_1$ and $g_2$ are continuous functions such that $g_1 (y) - g_1 \circ \varphi (y) = f (y) = g_2 (y) - g_2 \circ \varphi (y)$, then $(g_1 - g_2) (y) = (g_1 - g_2) (\varphi (y))$ is constant on each orbit.
Since the orbits are dense and $g_1$ and $g_2$ are continuous, they coincide up to a constant.
The axiom of choice guarantees the existence of at least one compact minimal set $N$ \cite[Appendix]{gottschalk:td55}.
If $\varphi$ is minimal, that is, every orbit is dense, and $M$ is compact, then one can choose $N = M$ in the theorem.
Without the assumption of compactness of $N$, it is only true that boundedness of the sequence of partial sums $\sum_{i = 0}^{k - 1} f_\varphi \circ \varphi^k (y)$ implies that $f$ is a coboundary, i.e.\ there exists a continuous function $g$ on $N$ that satisfies the relation $f = g - g \circ \varphi$, but not the converse \cite{browder:itn58,mccutcheon:ght99}.

\begin{cor}
Suppose $\varphi \in \Diff_+ (M,\sigma)$, and there exists a compact minimal set $N \subset M$, such that for some $x \in N$ the sequence $D_\tau ((\varphi^{-1})^k) (x)$ is unbounded.
Then $\varphi \notin \Diff (M,\tau')$ for any $\tau' \in \sigma$, and consequently, $\Diff_+ (M,\sigma)$ is properly essential.
In particular, if $x \in M$ is a periodic point of $\varphi$, that is $\varphi^m (x) = x$ for some integer $m \ge 1$, and $D_\tau ((\varphi^{-1})^m) (x) \not= 0$, then $\varphi \notin \Diff (M,\tau')$ for any $\tau' \in \sigma$, and $\Diff_+ (M,\sigma)$ is properly essential.
\end{cor}

The last part of the corollary in the case $m = 1$ is precisely Lemma~\ref{lem:criterion}.

\begin{proof}
Assume $\varphi \in \Diff (M,\tau')$, where $\tau' = e^g \tau$ for a smooth function $g$ on $M$.
Then by equation (\ref{eqn:criterion}), $f = g - g \circ \varphi$ on $M$.
By the theorem, the sequence $D_\tau ((\varphi^{-1})^k) (x)$ is bounded, a contradiction.
To prove the last statement, note that the orbit $\O (x)$ of $x$ is a compact minimal set, and the sequence $D_\tau ((\varphi^{-1})^k) (x)$ is bounded if and only if $D_\tau ((\varphi^{-1})^m) (x) = 0$.
\end{proof}

It would be interesting to have a necessary and sufficient condition for the proper essentiality of the conformal group $\Diff_+ (M,\sigma)$ in terms of Banyaga's conformal invariant, such as a smooth and global version of the Gottschalk-Hedlund Theorem, where $\varphi$ is a diffeomorphism, $f$ is a smooth function, and $g$ is also smooth and satisfies $f = g - g \circ \varphi$ on all of $M$.
The question of existence and regularity of solutions to various flavors of cohomological equations, such as Liv\v{s}ic theory or Liv\v{s}ic regularity, and other bounded implies coboundary type results, is studied in a myriad of papers in the literature.

We remark however that if $g$ is a function on $M$ that satisfies $f = g - g \circ \varphi$ on all of $M$, smoothness of $f$ alone is not enough to guarantee the  smoothness of $g$.
Indeed, for an irrational number $\theta$, choose a sequence of integers $n_k \ge 2^k$,  $k \ge 1$, such that $0 < n_k \theta - [ n_k \theta ] \le 2^{- n_k}$, where $[x]$ as usual denotes the greatest integer less than or equal to $x$, and define $n_k = - n_{- k}$ for $k < 0$.
Then the (real) function $f \colon S^1 \to \R$ defined by
	\[ f (e^{2 \pi i t}) = \sum_{j \not= 0} \frac{1}{j^2} \left( 1 - e^{2 \pi i n_j \theta} \right) \, e^{2 \pi i n_j t}, \]
is smooth, and the (real) function $g \colon S^1 \to \R$ defined by
	\[ g (e^{2 \pi i t}) = \sum_{j \not= 0} \frac{1}{j^2} e^{2 \pi i n_j t}, \]
is only $C^0$ but not $C^1$.
In fact, the example can easily be modified so that the function $f$ is smooth, and $g$ is $C^k$ but not $C^{k +1}$ for any $k \ge 0$, or even $L^2$ but not $C^0$.
On the other hand, it is immediate to check that $g (x) - g (e^{2 \pi i \theta} x) = f (x)$.
See \cite{furstenberg:set61, mueller:hvf11} for details.

\section{Diffeomorphisms reversing the orientation of a conformal structure} \label{sec:orientation}
In this section, we allow the function $f$ in the equation $\varphi^* \tau = f \tau$ to be any smooth function that vanishes nowhere, and drop the subscript $+$ wherever applicable.
For example, we define
	\[ C_\ad^\infty (M,\tau) = \{ f \in C^\infty (M) \mid f \tau \in \sigma \} \]
the set of nowhere vanishing functions that can appear as conformal rescaling factors of $\tau$.
Again if $\tau' = g \tau$, there is a natural $1-1$ correspondence between $C_\ad^\infty (M,\tau)$ and $C_\ad^\infty (M,\tau')$ given by $f \mapsto f / g$.
For simplicity, we assume $M$ is connected, so that $f$ is either everywhere positive (as in the rest of the paper) or everywhere negative.
If $\tau$ is a contact form or a volume form, then $C_\ad^\infty (M,\tau)$ consists of all nowhere zero functions, whereas if $\tau$ is a symplectic form and $\dim M \ge 4$, then by the closedness condition, $C_\ad^\infty (M,\tau)$ is comprised of nonzero constant functions only.
Similarly we define $\Diff (M,\sigma) \supset \Diff_+ (M,\sigma)$ the automorphism group of the \emph{unoriented} conformal structure $\sigma$ by
	\[ \Diff (M,\sigma) = \{ \varphi \in \Diff (M) \mid \varphi^* \tau = f \tau \ {\rm for \ some} \ f \in C_\ad^\infty (M,\tau) \}. \]
From now on, assume $\sigma$ is one of the classical conformal structures.
We note that if $\tau \in \sigma$, then $- \tau \in \sigma$, and $\Diff (M,\tau) = \Diff (M,- \tau)$.
Thus $\Diff (M,\sigma)$ is essential or properly essential, respectively, whenever $\Diff_+ (M,\sigma)$ is.

Let $\tau \in \sigma$, and $\varphi \in \Diff (M,\sigma)$ such that $\varphi^* \tau = f \tau$ for some everywhere positive or negative function $f$ on $M$.
If $\tau' = g \tau \in \sigma$ is any other tensor field in the same (unoriented) conformal class, then $\varphi^* (g \tau) = (g \circ \varphi) f \tau$, and thus $\varphi$ preserves the tensor field $\tau'$ if and only if $(g \circ \varphi) f = g$.
Since $g \circ \varphi$ everywhere on $M$ has the same sign as $g$, we have the following result.

\begin{lem}
If there exists a conformal diffeomorphism $\varphi \in \Diff (M,\sigma)$ such that $\varphi^* \tau = f \tau$ for some $\tau \in \sigma$ and a \emph{negative} function $f$, then $\Diff (M,\sigma)$ is properly essential.
Otherwise $\Diff (M,\sigma)$ is essential or properly essential, respectively, if and only if $\Diff_+ (M,\sigma)$ is.
\end{lem}

For the automorphism group of an unoriented conformal symplectic structure on a compact manifold $M^{2 n}$ with $n > 1$, this lemma can be reformulated as follows.

\begin{lem}
The group of (not necessarily orientation-preserving) conformal symplectic diffeomorphisms of a compact symplectic manifold $(M,\omega)$ of dimension $2 n \ge 4$ consists of all (anti-)symplectic diffeomorphisms $\varphi$, that is, $\varphi^* \omega = \omega$ or $- \omega$.
It is properly essential, and in particular essential, if and only if $(M,\omega)$ admits a symplectic involution $\varphi$, i.e.\ $\varphi^* \omega = - \omega$.
\end{lem}

This holds for example for $\C P^n$ with the Fubini-Study symplectic form $\omega_{\rm FS}$, since the reflection in $\R P^n \subset \C P^n$, $z \mapsto \overline{z}$, reverses the sign of $\omega_{\rm FS}$.
This also reverses the sign of the induced volume form $\omega_{\rm FS}^n$ if and only if $n$ is odd.
For $T^{2 n}$ with any symplectic structure $\omega = \sum_i c_i dx_i \wedge dy_i$ with constant coefficients (e.g.\ the standard symplectic structure), the diffeomorphism $(x, y) \mapsto (y, x)$ reverses the sign of $\omega$.

\section{Some final remarks}

\begin{lem}
The subset 
	\[ \bigcup_{\tau \in \sigma} \Diff (M,\tau) \subset \Diff_+ (M,\sigma) \]
is closed under inverses and conjugation by conformal diffeomorphisms in $\Diff_+ (M,\sigma)$.
Thus this subset forms a normal subgroup of $\Diff_+ (M,\sigma)$ if and only if it is closed under composition.
\end{lem}

\begin{proof}
Suppose $\tau \in \sigma$, $\varphi \in \Diff (M,\tau)$, and $\psi \in \Diff_+ (M,\sigma)$.
Then $\varphi^* \tau = \tau$, and $\psi^* \tau = e^f \tau$ for some smooth function $f$ on $M$.
$\psi^{-1} \circ \varphi \circ \psi$ preserves the tensor field $e^f \tau \in \sigma$.
The remaining statements are obvious.
\end{proof}

\begin{cor}
Suppose the conformal group $\Diff_+ (M,\sigma)$ is properly essential.
Then if the set $\cup_{\tau \in \sigma} \Diff (M,\tau)$ is closed under composition, $\Diff_+ (M,\sigma)$ is not a simple group.
\end{cor}

The same argument applies to the automorphism groups of unoriented conformal structures, or the identity components of any of the above groups.
If $M$ is a closed and connected contact manifold, then $\Diff_0 (M,\xi)$ is a simple group according to T.~Rybicki~\cite{rybicki:cc10}, and thus the union of $\Diff_0 (M,\alpha)$ over all contact forms $\alpha$ with $\ker \alpha = \xi$ can never form a group.
On the other hand, the group generated by the latter equals $\Diff_0 (M,\xi)$ by the same argument, or more precisely, any contact diffeomorphism in $\Diff_0 (M,\xi)$ can be written as a finite composition of diffeomorphisms, each of which preserves some contact form (that varies for different factors) with kernel $\xi$.
In other words, if $\varphi^* \alpha = e^f \alpha$ for a contact diffeomorphism $\varphi$, isotopic to the identity in the group of contact diffeomorphisms, then ${\varphi = \varphi_k \circ \cdots \circ \varphi_1}$  for some integer $k$, smooth functions $f_i$, and contact diffeomorphisms $\varphi_i$ preserving the contact forms $e^{f_i} \alpha$ (and isotopic to the identity in the group of diffeomorphisms preserving $e^{f_i} \alpha$), and $f$ can be written
	\[ f = \sum_{i = 1}^k (f_i - f_i \circ \varphi_i) \circ \varphi_{i - 1} \circ \cdots \circ \varphi_1. \]
The next statement follows immediately from the preceding discussion.

\begin{pro}
The minimal number of factors in the above decomposition of a contact diffeomorphism is a contact invariant (it is conjugation invariant), and this invariant is nontrivial on $\Diff_0 (M,\xi)$.
\end{pro}

\section*{Acknowledgements}
We would like to thank Urs Frauenfelder for suggesting the use of open books for the proof of Lemma~\ref{lem:closed-orbit}, and Otto van Koert and Klaus Niederkr\"uger for useful discussions related to the proof of Lemma~\ref{lem:standard}.

\bibliography{essential}
\bibliographystyle{amsplain}
\end{document}